\documentclass[11pt,reqno]{amsart}
\usepackage[body={7in,9in},centering]{geometry}
\usepackage{amsmath,amssymb,amsthm,mathrsfs}
\usepackage{color,xcolor}
\definecolor{cobalt}{RGB}{61,99,181}
\usepackage[colorlinks,citecolor=cobalt,linkcolor=cobalt]{hyperref}
\usepackage{bbding}
	\newtheorem{theorem}{Theorem}[section]

\newtheorem{lemma}[theorem]{Lemma}	
\newtheorem{definition}[theorem]{Definition}

\numberwithin{equation}{section}

\date{\today}

\makeatletter

\newcommand{\Rmnum}[1]{\expandafter\@slowromancap\romannumeral #1@}
\makeatother

\begin{document}\bibliographystyle{plain}

\title{Entropy degeneracy and entropy explosion of flows}

\author[Mengjie Zhang]{Mengjie Zhang\textsuperscript{1}}
\address{\textsuperscript{1} College of Mathematics and Statistics, Chongqing University, Chongqing, 401331, P. R. China}
\email{202006021044t@stu.cqu.edu.cn}



	\begin{abstract}
	In this paper, we prove that entropy degeneracy and entropy explosion happens in the flow constucted by Ohno. We also constucte a flow which has the only one invariant and ergodic measure supporting at a fixed point. This flow is no entropy explosion.
\end{abstract} \maketitle

\section{Introduction}
In 1980, Ohno constructed a flow \cite {1980A}, show that the topological equivalence of the two flow can not preserve zero entropy. In 2009, Sun,Young and Zhou constructed an example \cite{2009Topological}, pointing out that the above phenomenon exists even in ${C^r}$ differential manifolds, that is, there are two topologically equivalent ${C^r}$ flows that do not preserve zero entropy. This example illustrates that it is not the differential regularity and spatial topology of the flow, but the temporal parameterization of the flow that causes the singular change in entropy of the flow. In 2011, Sun and Zhang constructed the extremely extreme case \cite{2017Zero}, two topologiously equivalent flows, one with 0 entropy and one with $+ \infty $. The singularity described above occurs only between two flows that are topologiously equivalent with a fixed point, while two flows that are topologiously equivalent without a fixed point maintain 0 entropy and $+ \infty $\cite{1990Topological,2007Entropy,Bowen1972ExpansiveOF}.

We will further study the entropy change of two topologically equivalent flows. If a flow with topological entropy \cite{1969On} greater than zero is transformed into a flow with topological entropy equal to zero through the topological equivalence relation, we call this phenomenon entropy degeneracy. If it is transformed into a flow with increasing topological entropy through topological equivalence, we call it entropy explosion \cite{2008ENTROPY}. For the standard  suspension  flow constructed by Ohno, we prove that any given real number $b \in \left[{0,\left. {+ \infty} \right)} \right.$  a flow topologically equivalent to the standard  suspension  flow can be found and  the topological entropy of this flow is equal to b. This indicates that  entropy degeneracy and entropy explosion can occur in the standard  suspension  flow .

The structure of this paper is as follows: In Section 2, preparatory knowledge is introduced; In Section 3, it is shown that the flow constructed by Ohno can have entropy degradation and entropy explosion at the same time, which are Theorems A and B; In Section 4, a flow in which entropy explosion cannot occur is constructed, and the main results are Theorems C and D.

\section{Preliminary}
Throughout this article, we have agreed that $\left({X,d} \right)$ is a compact metric space and ${\mathscr B}\left(X \right)$ is a Borel $\sigma$ algebra. The following definitions and theorems can be found in reference \cite{2017Differential, Perko2001Differential, R1987Ergodic, 2008TIME} .

\begin{definition}
	$\phi$ denotes a flow on $\left( {X,d} \right)$ , that is, $\phi$ is a continuous mape $\phi :{\rm X} \times {\rm{R}} \to {\rm X}$ satisfying: 
	
	(1) $\phi \left( {x,0} \right) = x$, for all $x \in X$;
	
	(2) $\phi \left( {\phi \left( {x,s} \right),t} \right) = \phi \left( {x,s + t} \right)$, for all $x \in X$ and $s,t \in {\rm{R}}$.
\end{definition}
For the sake of narrative convenience, denote  (2) as ${\phi _t} \circ {\phi _s} = {\phi _{t + s}}$ . For all $t \in {\rm{R}}$ , ${\phi _t} = \phi \left( { \cdot ,t} \right):X \to X$  is continuous and has a continuous inverse ${\phi _{ - t}}$. Therefore, ${\phi _{  t}}$ is a homemorphism.
\begin{definition}\label{def:one}
	Let $\phi :X \times {\rm{R}} \to X$ be a flow on a compact metric space $X$ and measure $\mu $ is $\phi$-invariant. Definite the theoritical entropy and the topological entropy of a flow $\phi$ with respect to the measure$\mu $ , they are respectively:$h\left( \phi  \right) = \mathop {\sup }\limits_{\mu  \in {{\rm M}_{erg,\phi }}} {h_\mu }\left( \phi  \right)$ and $h\left( \phi  \right) = h\left( {{\phi _1}} \right)$ .	
\end{definition}

${{\rm M}_{erg,\phi }}$ denote the set of all invariant ergodical Borel probability measures for the flow $\phi$ .
\begin{theorem}[The variational principle]
	Let $\phi :X \times {\rm{R}} \to X$ be a flow on a compact metric space $X$. Then
	\begin{center}
		$h\left( \phi  \right) = \mathop {\sup }\limits_{\mu  \in {{\rm M}_{erg,\phi }}} {h_\mu }\left( \phi  \right)$.
	\end{center}	
	According to the definition of the entropy of the flow, it follows
	\begin{center}
		$h\left( \phi  \right) = \mathop {\sup }\limits_{\mu  \in {{\rm M}_{erg,\phi }}} {h_\mu }\left( {{\phi _1}} \right) = \mathop {\sup }\limits_{\mu  \in {{\rm M}_{erg,{\phi _1}}}} {h_\mu }\left( \phi  \right)$.
	\end{center}	
\end{theorem}

\begin{definition}
	Let $X,Y$ be two compact metric spaces and $\phi :X \times \rm{R} \rightarrow X$ , $\psi :Y \times \rm{R} \rightarrow Y$ be two flows. If there is a preserving time orientation homeomorphism $\pi :X \rightarrow Y$ ,
	 satisfying for all $t \in {\rm{R}}$ ,
	\begin{center}
		$\left\{ {{\phi _t}\left( x \right)\left| {t \in {\rm{R}}} \right.} \right\} = \left\{ {{\pi ^{ - 1}}{\psi _t}\pi \left( x \right)\left| {t \in {\rm{R}}} \right.} \right\}$.
	\end{center}
	$\pi$ maps the orbits $Orb\left( {x,\phi } \right) = \left\{ {{\phi _t}\left( x \right)\left| {t \in {\rm{R}}} \right.} \right\}$  of $\phi$ in $x$ onto the orbits $Orb\left( {\pi \left( x \right),\psi } \right) = \left\{ {{\psi _t}\left( {\pi x} \right)\left| {t \in {\rm{R}}} \right.} \right\}$  of $\psi$ in $\pi \left( x \right) \in Y$ and preserves their time orientation. Then we call that two flows are topological equivalent.
	
	In particular, if it is satisfied that for all $x \in X$ and $t \in {\rm{R}}$, we have $\pi \circ {\phi _t}\left(x \right) = {\psi _t}\ circ \pi \left(x \right)$.The following exchange diagram is obtained:
	\begin{center}
		$ \begin{array}{*{20}{c}}
			X&{\mathop  \to \limits^{{\phi _t}} }&X\\
			{ \downarrow \pi }&{}&{ \downarrow \pi }\\
			Y&{\mathop  \to \limits^{{\psi _t}} }&Y
		\end{array}$.	
	\end{center}
\end{definition}

\begin{definition}
	Let $\phi :X \times {\rm{R}} \to X$ be a flow on the compact metric space $X$. Let $X$ be decomposed into three disjoint $\phi$-invariant sets  $A$ , $B$ and $N$ such that $\mu \left( A \right) > 0$ and $\mu \left( N \right) = 0$. Let $\theta \left( {x,t} \right)$ be a real measurable function defined on $\left( {X\backslash N} \right) \times \left( { - \infty , + \infty } \right)$ with the following propertier: for every fixed $x \in A$ ,
	
	(1) $\theta \left( {x,t} \right)$ is finite for all $t \in {\rm{R}}$;
	
	(2) $\theta \left( {x,t + s} \right) = \theta \left( {x,t} \right) + \theta \left( {\phi \left( {x,t} \right),s} \right)$ for all $t,s \in {\rm{R}}$;
	
	(3) $\theta \left( {x,t} \right)$ is continuous and nondescreasing in $t$;
	
	(4)  $\theta \left( {x,0} \right) = 0$,
	$\mathop {\lim }\limits_{t \to  + \infty } \theta \left( {x,t} \right) =  + \infty $,
	$\mathop {\lim }\limits_{t \to  - \infty } \theta \left( {x,t} \right) =  - \infty $;
	
	(5) $\theta \left( {x,0} \right) = 0$, for all $x \in B$, $t \in {\rm{R}}$.\\
	$\theta \left( {x,t} \right)$ is called an additive function of $\phi$ with the set $A$.
\end{definition}

\begin{definition}
	Let $\phi$ be a flow on Borel probability space $\left( {X,{\mathscr B},\mu } \right)$ , and let $\theta \left( {x,t} \right)$ be an additive function of $\phi$ with the set $A$.
	
	Define ${\hat \phi _t}x = {\phi _{\tau \left( {x,t} \right)}}x$, where $\tau \left( {x,t} \right) = \sup \left\{ {s|\theta \left( {x,s} \right) \le t} \right\},\ \ \forall x \in A, \ \forall t \in{\rm{R}} $.
	
	Define the regular set $\hat X$ of the additive function $\theta \left( {x,t} \right)$, $\hat X = \left\{ {x \in A|\theta \left( {x,t} \right) > 0,\forall t > 0} \right\}$.
\end{definition}

\begin{theorem} \rm{\cite{2008TIME}} \label{thm:change}
	Let $\hat {\mathscr B} = {\mathscr B}\cap \hat X$ . Then $\hat \phi :\hat X \times {\rm{R}} \to \hat X$ is a flow, called as the time-changed flow of $\phi$ with respect to $\theta \left( {x,t} \right)$ .
\end{theorem}

\begin{theorem}\rm{\cite{1980A}}
	Let $\phi$ and $\psi$ be two topological equivalent flows on the compact metric space $X$. Let $\pi$ be a perserving time orientation homeomorphism such that $\pi \left( {Orb\left( {x,\phi } \right)} \right) = Orb\left( {\pi x,\psi } \right)$ for all $x \in X$. Let ${X_0}$ be a set containing all fixed points of  $\phi$. Then there exists $\theta :X\backslash {X_0} \times {\rm{R}} \to {\rm{R}}$, such that
	
	(1) $\theta \left( {x,0} \right) = 0$, $\theta \left( {x, \cdot } \right)$ is strictly increasing for all $x \in X\backslash {X_0}$;
	
	(2) $\theta \left( {x,t + s} \right) = \theta \left( {x,t} \right) + \theta \left( {\phi \left( {x,t} \right),s} \right)$, for all $x \in X\backslash {X_0}$, $s,t \in {\rm{R}}$;
	
	(3) $\pi  \circ {\phi _t}\left( x \right) = {\psi _{\theta \left( {x,t} \right)}} \circ \pi \left( x \right)$, for all $x \in X\backslash {X_0}$, $t \in {\rm{R}}$.
\end{theorem}

According to Theorem 2.8, $\theta $ is an additive function. The time-changed flow of $\phi $ with respect to $\theta $ is ${\hat \phi _t} = {\pi ^ - } \circ {\psi _t} \circ \pi $.

\section{Entropy degeneracy and entropy explosion of flows}
The specific construction process of the Ohno example is given in section 3.1, and the main theorems and their proofs in this chapter are given in section 3.2.

\subsection{Construction of Ohno example}

For the convenience of this article, we will review Ohno's construction here, which can be found in the references\cite{1980A}.

Give a discrete topology to the set $\left\{ {0,1} \right\}$.  Note $\sum \left( 2 \right) = \mathop \prod \limits_{ - \infty }^{ + \infty } {\left\{ {0,1} \right\}^{\rm{Z}}}$ and give the product topology. Considering the shifting homeomorphism $T:\sum \left( 2 \right) \to \sum \left( 2 \right)$, i.e. $\left( {{x_i}} \right) \to \left( {{x_{i + 1}}} \right)$ .

Step 1, Selecting a bidirectional infinite sequence ${x^ * } = \left( {{x^ * }\left( i \right)} \right) \in \sum \left( 2 \right)$, which satisfying the following proportion:

(1) For all $k \in {\rm{Z}}$ and $n \ge 1$, ${x^ * }\left( k \right){x^ * }\left( {k + 1} \right) \cdots {x^ * }\left( {k + 4 \cdot {3^n}} \right) \succ {I_n} = \underbrace {11 \cdots 1}_{2n - 1}$, i.e. the sequence ${x^ * }\left( k \right){x^ * }\left( {k + 1} \right) \cdots {x^ * }\left( {k + 4 \cdot {3^n}} \right)$ contains a subsequence which are $2n - 1$ consecutive 1's.

(2) For all $n \ge 2$, the sequence ${x^ * }$ contains ${2^{{p_n}}}$ finite subsequences of length $2 \cdot {3^{n - 1}}$. Here, $p{}_n = \frac{{{3^{n - 1}} + 1}}{2}$.

Let $\alpha $ be a symbol which will be replaced by 0 or 1. Define ${x_1} = 1\alpha $ and ${\tilde x_1} = 11$. We change one $\alpha $ in ${x_1}$ to 1. Define ${x_2} = {x_1}{\tilde x_1}{x_1} = 1\alpha 111\alpha $. We get ${\tilde x_2}$ by changing one $\alpha $ in ${x_2}$ to 1 with ${\tilde x_2} \succ {I_2}$. That is the sequence ${\tilde x_2}$ contains a subsequence which are $2 \cdot 2 - 1 = 3$ consecutive 1's. For example, ${\tilde x_2} = 11111\alpha $. Define ${x_3} = {x_2}{\tilde x_2}{x_2}$. We get ${\tilde x_3}$ by changing one $\alpha $ in ${x_3}$ to 1 with ${\tilde x_3} \succ {I_3}$. In general, put ${x_n} = {x_{n - 1}}{\tilde x_{n - 1}}{x_{n - 1}}$. Since each nontrailing $\alpha $of the sequence ${x_n}$ and ${\tilde x_{n-1}} \succ {I_{n-1}}$ have 1 on both sides and one of the $\alpha $in ${x_n}$can be replaced by 1, we get ${\tilde x_n}$ with ${\tilde x_n} \succ {I_n}$. Using induction, we obtain $\left\{ {{x_n}} \right\}_1^\infty $. For all $n \ge 1$, ${x_n}$ and ${\tilde x_n}$ are the length of $2 \cdot {3^{n - 1}}$. And the sequence ${\tilde x_{n + 1}}$ before $2 \cdot {3^{n - 1}}$ on the position is ${x_n}$. So the following limit exists:${x^ + } = \mathop {\lim }\limits_{n \to \infty } {x_n}$. Then ${x^ + } \in {\left\{ {1,\alpha } \right\}^{\rm{N}}}$.

As is known by the constructor, every subsequence of ${x^ + }$ with the length of $4 \cdot {3^n} = 2 \cdot 2 \cdot {3^{n - 1}} \cdot 3$ either contains ${\tilde x_n}$, which in turn contains ${I_n}$; or contains a modified ${\tilde x_n}$, that modifies several $\alpha $ in ${\tilde x_n}$ to 1 such that contains ${I_n}$. Therefore, for all $k \ge 1$ ,$n \ge 2$ , we have 
\begin{center}
	${x^ + }\left( k \right){x^ + }\left( {k + 1} \right) \cdots {x^ + }\left( {k + 4 \cdot {3^n}} \right) \succ {I_n}$.
\end{center}

On the other hand, the number of $\alpha $ in ${x_1} = 1\alpha $ can be written as $1 = \frac{{{3^{1 - 1}} + 1}}{2} = {p_1}$; the number of $\alpha $ in ${x_2} = 1\alpha 111\alpha $ can be written as $2 = \frac{{{3^{2 - 1}} + 1}}{2} = {p_2}$. Let the number of $\alpha $ in ${x_{n - 1}}$ be written as $\frac{{{3^{n - 2}} + 1}}{2} = {p_{n - 1}}$. Noting ${x_n} = {x_{n - 1}}{\tilde x_{n - 1}}{x_{n - 1}}$, then the number of $\alpha $ in ${x_n}$ is
\begin{center}
	$\frac{{{3^{n - 2}} + 1}}{2} + \left( {\frac{{{3^{n - 2}} + 1}}{2} - 1} \right) + \frac{{{3^{n - 2}} + 1}}{2} = \frac{{{3^{n - 1}} + 1}}{2} = {p_n}$.
\end{center}
Replacing $\alpha $ by 0 or 1 in ${x_n}$, we obtain ${2^{p{}^n}}$ different ${x_n}$ with the length of $2 \cdot {3^{n - 1}}$. Accordingly, we make the following amendment to ${x^ + }$. \\
Starting from the left end of ${x^ + }$, we pick two ${x_1}$ because there are infinite ${x_1}$ in ${x^ + }$. Changing one $\alpha $ in ${x_1}$ to 0, one $\alpha $  to 1, we obtain ${2^{{p_1}}} = 2$ different ${x_1}$ with the length of 2. Take the tail end of the rightmost ${x_1}$, and replace the remaining the left side  $\alpha $ to 1.
Starting from the left end of ${x^ + }$, we pick ${2^{{p_2}}} = 4$ ${x_2}$ with the length of $2 \cdot {3^{2 - 1}} = 6$. 
Changing these $\alpha $ in ${x_2}$ to 0 or 1, we get four different subsequence with the length of 6.  Repeating this procedure infinitely gives the corrected unilateral infinite sequence, still denoted as ${x^ + }$. We point out that the corrected ${x^ + }$ belongs to ${\left\{ {0,1} \right\}^{\rm{N}}} = \mathop \prod \limits_0^\infty  \left\{ {0,1} \right\}$. ${x^ + }$ satisfies the conditions (1) and (2) proposed above.

Define
\begin{center}
	${x^ * }\left( k \right) = \left\{ {\begin{array}{*{20}{c}}
			{{x^ + }\left( k \right),}\\
			{0,}\\
			{{x^ + }\left( { - k} \right),}
		\end{array}\begin{array}{*{20}{c}}
			{k \ge 1}\\
			{k = 0}\\
			{k \le  - 1}
	\end{array}} \right.$.
\end{center}
Then ${x^ * }\left( k \right) \in \sum \left( 2 \right)$  and satisfies the conditions (1) and (2).

Step 2, Constructing a subsystem of symbology and discussing its topological entropy.

Let $X = \overline {Orb\left( {{x^ * },T} \right)} $. Then $X$ is the closure of ${x^ * }$ under $T$ iteration, and is compact $T$-invariant subset. $T:X \to X$ constitutes a subsystem of symbology, $E\left( {X,T} \right) \ne \emptyset $, where $E\left({X,T}\right)$ represents the set of all ergodic measures of $T$.

$n \ge 1$, let
\begin{center}
	${\hat I_n} = \left\{ {x \in X|x\left( k \right) = 1, - n + 1 \le k \le n - 1} \right\}$.
\end{center}
${\chi _{{{\hat I}_n}}}$ denote the feature functions of ${\hat I_n}$.

Take $\mu  \in E\left( {X,T} \right)$. According to the Birkhoff's ergodic theorem, for  $\mu  - a.e.x \in X$ there exists 
\begin{center}
	$\mu \left( {{{\hat I}_n}} \right) = \int_X {{\chi _{{{\hat I}_n}}}} d\mu  = \mathop {\lim }\limits_{m \to \infty } \frac{1}{m}\sum\limits_{i = 0}^{m - 1} {{\chi _{{{\hat I}_n}}}\left( {{T^i}x} \right)} $.
\end{center}

Since $x$ is the aggregation point of $\left\{ {{T^n}{x^ * }|n \in {\rm{Z}}} \right\}$, it is known that $x$ also satisfies the condition (1) of step 1 according to the topology of $\sum \left( 2 \right)$. It follows 
\begin{center}
	$\mu \left( {{{\hat I}_n}} \right) \ge \frac{1}{{4 \cdot {3^n}}}$.
\end{center}
On the other hand, according to the condition (2) by step 1 we have
\begin{center}
	$h\left( T \right) = \mathop {\lim }\limits_{n \to \infty } \frac{1}{n}\ln {\theta _n}\left( X \right) \ge \mathop {\lim }\limits_{n \to \infty } \frac{1}{{2 \cdot {3^{n - 1}}}}\ln {2^{{p_n}}} = \frac{1}{4}\ln 2 > 0$,
\end{center}
where ${\theta _n}\left( X \right) = \# \left\{ {\left[ {{i_0} \cdots {i_{n - 1}}} \right]|\exists x \in X,s.t.{x_0} = {i_0}, \cdots ,{x_{n - 1}} = {i_{n - 1}}} \right\}$.

Step 3, Suspending $\left( {X,T} \right)$ into a flow.

The discrete system $\left( {X,T} \right)$ has a unique fixed point (denoted ) $1 \in X$, which is a bidirectional infinite sequence with each position 1. Let ${X_ * } = X\backslash \left\{ 1 \right\}$. Then ${X_ * }$ is a locally compact space. Let $\gamma :{X_ * } \to \left( {0, + \infty } \right)$  be a positive continuous function, which is called a roof function. Then $\left( {{X_ * },T} \right)$ can be suspended with $\gamma $ to get a flow. Specifically, establish an equivalence relationship on $\left\{ {\left( {x,u} \right)|0 \le u \le \gamma \left( x \right),x \in {X_ * }} \right\}$ 
\begin{center}
	$\left( {x,\gamma \left( x \right)} \right) \sim \left( {Tx,0} \right)$.
\end{center}
We obtain the quotient space $X_ * ^\gamma $. The flow ${\varphi ^\gamma }:X_ * ^\gamma  \times {\rm{R}} \to X_ * ^\gamma $ is defined as
\begin{center}
	$\varphi _t^\gamma \left( {x,u} \right) = \left( {x,u + t} \right),  - u \le t < \gamma \left( x \right) - u$.
\end{center}
In this case, $X_ * ^\gamma $ is the metric space of local compactness and $\varphi _t^\gamma $ is the flow defined on the space $X_ * ^\gamma $.

Let ${X^\gamma } = X_ * ^\gamma  \cup \left\{ \Delta  \right\}$ is the single-point compactification of $X_ * ^\gamma $. Then ${X^\gamma }$ is the compactness metric space. And $\left( {{x_n},{u_n}} \right) \to \Delta $ on the space ${X^\gamma }$ if and only if there is ${x_n} \to 1$ on the space $X$. Put $\varphi _t^\gamma \left( \Delta  \right) = \Delta $. For all $t \in {\rm{R}}$, we expand the flow $\varphi _t^\gamma $ on $X_ * ^\gamma $ to the flow on  ${X^\gamma }$, and denote it ${\Phi ^\gamma }$.

\begin{lemma}\label{lem:one}
	For any two positive continuous $\gamma $ and $\gamma '$, ${\Phi ^\gamma }$ and ${\Phi ^{\gamma '}}$ are topological equivalent to each other.
\end{lemma}

\begin{lemma}
	Assume ${\gamma _0} = \mathop {\inf }\limits_{x \in {X^ * }} \gamma \left( x \right) > 0$. For all non-trivial (not atomic measure of $\Delta $) $\bar \mu  \in {{\rm M}_{erg,{\Phi ^\gamma }}}$, there exists an non-trivial measure $\mu  \in E\left( {X,T} \right)$ (not atomic measure of $\left\{ 1 \right\}$ ), such that for all continuous function $f$ on ${X^\gamma }$,
	\begin{center}
		${E_{\bar \mu }}\left( f \right) = \frac{1}{{{E_\mu }\left( \gamma  \right)}}{E_\mu }\left( {\int_0^{\gamma \left( x \right)} {f\left( {x,t} \right)dt} } \right)$,
	\end{center}
	where
	\begin{center}
		${E_{\bar \mu }}\left( f \right) = \int {fd\bar \mu ,{E_\mu }\left( \gamma  \right) = \int {\gamma d\mu } } $,\\
		${E_\mu }\left( {\int_0^{\gamma \left( x \right)} {f\left( {x,t} \right)dt} } \right) = \int {\left( {\int_0^{\gamma \left( x \right)} {f\left( {x,t} \right)dt} } \right)d\mu } $.
	\end{center}
	
\end{lemma}

\begin{theorem}\label{thm:two}
	If for all non-trivial $\mu  \in E\left( {X,T} \right)$, ${E_\mu }\left( \gamma  \right) =  + \infty $. Then ${\delta _\Delta }$ is the only invariant ergodic measure of ${\Phi ^\gamma }$. That is, ${{\rm M}_{erg,{\Phi ^\gamma }}} = \left\{ {{\delta _\Delta }} \right\}$ .
\end{theorem}
Combining the above lemmas and inferences, Ohno determined that there are two flows, one with positive entropy and one with 0 entropy. Take
\begin{center}
	$\gamma \left( x \right) = \left\{ {\begin{array}{*{20}{c}}
			{n \cdot 4 \cdot {3^n},}\\
			{1,}
		\end{array}\begin{array}{*{20}{c}}
			{x \in {{\hat I}_n}\backslash {{\hat I}_{n + 1}},n \ge 1,}\\
			{x \in {X_ * }\backslash {{\hat I}_1}}
	\end{array}} \right.$, and $\gamma '\left( x \right) = 1$.
\end{center}
Easy to prove $\gamma \left( x \right)$ continuous.  By Lemma 3.1, $\left( {{X^\gamma },{\Phi ^\gamma }} \right)$ and $\left( {{X^{\gamma '}},{\Phi ^{\gamma '}}} \right)$ are topologically equivalent to each other.
On the one hand, there exists $h\left( {{\Phi ^{\gamma '}}} \right) = h\left( T \right) > 0$.
On the other hand, for all non-trivial measure $\mu  \in E\left( {X,T} \right)$, it holds that 
\begin{center}
	${E_\mu }\left( \gamma  \right) > \int_{{{\hat I}_n}} {n \cdot 4 \cdot {3^n}d\mu }  > n \cdot 4 \cdot {3^n}\mu \left( {{{\hat I}_n}} \right) > n, \forall n \in {\rm{N}}$.
\end{center}
So, ${E_\mu }\left( \gamma  \right) =  + \infty $. According to Theorem 3.3, it is known that ${\delta _\Delta }$ is the only invariant ergodic measure of  ${\Phi ^\gamma }$.  From the variational principle of the flow, we obtain $h\left( {{\Phi ^\gamma }} \right) = 0$.
\subsection{The proof of Theorem A and Theorem B}
For the following discussion, we write the flow ${\Phi ^{\gamma '}}$ as ${\Phi ^1}$. i.e. ${\Phi ^1}:{X^1} \times {\rm{R}} \to {X^1}$, where ${X^1} = X_ * ^1 \cup \left\{ \Delta  \right\}$. Now give the main results of the section 3 in this article.

\textbf{Theorem A.}
Known $h\left( {{\Phi ^1}} \right) = h\left( T \right) > 0$. For any real number $b \in \left( {0,h\left( T \right)} \right)$, we can find a flow $\left( {{X^\gamma },{\Phi ^\gamma }} \right)$ constructed with the roof function $\gamma \left( x \right)$, such that there exists $\nu  \in {{\rm M}_{erg,{\Phi ^\gamma }}}$ satisfying ${h_\nu }\left( {{\Phi ^\gamma }} \right) = b$.

\textbf{Theorem B.}
For any real number $0 < b <  + \infty $ , we can find a flow $\left( {{X^\gamma },{\Phi ^\gamma }} \right)$ constructed with the roof function $\gamma \left( x \right)$, such that $h\left( {{\Phi ^\gamma }} \right) = b$.

According to Theorem A and Theorem B, it is concluded that for the flow ${\Phi ^1}$ constructed above, both entropy degeneracy and entropy explosion can occur.

Before proving theorem A and theorem B, some lemmas are required.

\begin{lemma}\label{lem:shi}
	Let $a\left( x \right):{X_ * } \to \left( {0, + \infty } \right)$ be a positive finite continuous map. ${\Phi ^a}$ is a flow constructed with the roof function $a\left( x \right)$, i.e. ${\Phi ^a}:{X^a} \times {\rm{R}} \to {X^a}$, where ${X^a} = X_ * ^a \cup \left\{ {\Delta '} \right\}$. Put
	\begin{center}
		$a\left( {x,u} \right) = \left\{ {\begin{array}{*{20}{c}}
				{a\left( x \right),}&{\forall x \in {X_ * },0 \le u < 1}\\
				{c,}&\Delta
		\end{array}} \right.$,
	\end{center}
	where $c$ is a constant greater than zero. Define an additive function with $a\left( {x,u} \right)$, 
	\begin{center}
		$\theta \left( {x,u,t} \right) = \int_0^t {a\left( {{\Phi ^1}_s\left( {x,u} \right)} \right)} ds$.
	\end{center}
	$\left( {{X^1},\hat \Phi } \right)$ denote the time-changed flow of $\left( {{X^1},{\Phi ^1}} \right)$ with the additive function $\theta \left( {x,u,t} \right)$.
	
	Then $\left( {{X^1},\hat \Phi } \right)$ and $\left( {{X^a},{\Phi ^a}} \right)$ are topologically conjugate each other.
\end{lemma}

\begin{proof}
	Define a map $\pi :{X^1} \to {X^a}$, with 
	\begin{center}
		$\pi\left( {x,u} \right) = \left\{ {\begin{array}{*{20}{c}}
				{x\left( ua\left(x\right) \right),}&{\forall x \in {X_ * },0 \le u < 1}\\
				{\Delta '}&\Delta
		\end{array}} \right.$.
	\end{center}
	$\pi $ is continuous and there is a continuous inverse map ${\pi ^ - }$, so $\pi $ is a homeomorphism because of the continuity.
	
	Take $\tau \left( {x,u,s} \right){\rm{ = }}\sup \left\{ {\left. t \right|\theta \left( {x,u,t} \right) \le s} \right\}$. We have $\theta \left( {x,0,\tau \left( {x,0,s} \right)} \right) = s$, $\tau \left( {x,0,0} \right) = 0$.
	The following only needs to prove that for all $\left( {x,v} \right) \in {X^a}$ and $t \in {\rm{R}}$, it holds   $\pi  \circ {\hat \Phi _t} \circ {\pi ^ - }\left( {x,v} \right) = {\Phi ^a}_t\left( {x,v} \right)$.
	
	(1) For all $\left( {x,v} \right) \in {X^a}$, when $ - v \le t <  - v + a\left( x \right)$: since
	\begin{center}
		${\hat \Phi _t} \circ {\pi ^ - }\left( {x,v} \right) = {\hat \Phi _t}\left( {x,\tau \left( {x,0,v} \right)} \right) = {\hat \Phi _{t + v}}\left( {x,\tau \left( {x,0,0} \right)} \right)$,
	\end{center}
	it follows that
	\begin{equation*}
		\begin{split}
			\pi  \circ {{\hat \Phi }_t} \circ {\pi ^ - }\left( {x,v} \right) &= \pi  \circ {{\hat \Phi }_t}\left( {x,\frac{v}{{a\left( x \right)}}} \right)\\
			&  = \pi  \circ {{\hat \Phi }_{t + v}}\left( {x,0} \right)\\
			&  = \pi \left( {x,\frac{{t + v}}{{a\left( x \right)}}} \right)\\
			&  = \left( {x,t + v} \right)\\
			&  = {\Phi ^a}_t\left( {x,v} \right).
		\end{split}
	\end{equation*}
	For all $\left( {x,v} \right) \in {X^a}$, if $t = a\left( x \right)$, we have $\pi  \circ {\hat \Phi _t} \circ {\pi ^ - }\left( {x,0} \right) = \left( {Tx,0} \right) = {\Phi ^a}_t\left( {x,0} \right)$ .
	
	(2) Denote ${c^ + }\left( {n,x} \right) = \sum\nolimits_{k = 0}^n {a\left( {{T^k}x} \right)} $ and ${c^ - }\left( {n,x} \right) = \sum\nolimits_{k = 1}^n {a\left( {{T^{ - k}}x} \right)} $. When ${c^ + }\left( {x,n - 1} \right) \le t + v < {c^ + }\left( {x,n} \right)$, we obtain
	\begin{equation}\label{equ:one}
		\begin{split}
			{\Phi ^a}_t\left( {x,v} \right)& = \left( {Tx,v + t - a\left( x \right)} \right) \\
			&=  \left( {{T^2}x,v + t - a\left( x \right) - a\left( {Tx} \right)} \right)\\
			\vdots \\
			&= \left( {{T^n}x,v + t - {c^ + }\left( {x,n - 1} \right)} \right)
		\end{split}.
	\end{equation}
	Since
	\begin{align*}
		&\int\limits_0^{1 - \frac{v}{{a\left( x \right)}} + \left( {n - 1} \right) + \frac{{t + v - {c^ + }\left( {x,n - 1} \right)}}{{a\left( {{T^n}x} \right)}}} {a\left( {{\Phi ^1}_s\left( {x,\frac{v}{{a\left( x \right)}}} \right)} \right)ds} \\
		&=\left( {1 - \frac{v}{{a\left( x \right)}}} \right)a\left( x \right) + a\left( {Tx} \right) +  \cdots  + a\left( {{T^{n - 1}}x} \right) + a\left( {{T^n}x} \right)\frac{{t + v - {c^ + }\left( {x,n - 1} \right)}}{{a\left( {{T^n}x} \right)}}\\
		&= t,
	\end{align*}
	it follows that
	\begin{align*}
		\tau \left( {x,\frac{v}{{a\left( x \right)}},t} \right) = 1 - \frac{v}{{a\left( x \right)}} + \left( {n - 1} \right) + \frac{{t + v - {c^ + }\left( {x,n - 1} \right)}}{{a\left( {{T^n}x} \right)}}.
	\end{align*}
	Therefore,  
	\begin{equation}\label{equ:two}
		\begin{split}
			\pi  \circ {{\hat \Phi }_t} \circ {\pi ^ - }\left( {x,v} \right)& = \pi  \circ {{\hat \Phi }_t}\left( {x,\frac{v}{{a\left( x \right)}}} \right)\\
			& = \pi  \circ {\Phi _{\tau \left( {x,\frac{v}{{a\left( x \right)}},t} \right)}}\left( {x,\frac{v}{{a\left( x \right)}}} \right)\\
			& = \pi \left( {x,\frac{v}{{a\left( x \right)}} + \tau \left( {x,\frac{v}{{a\left( x \right)}},t} \right)} \right)\\
			& = \pi \left( {x,\frac{v}{{a\left( x \right)}} + 1 - \frac{v}{{a\left( x \right)}} + \left( {n - 1} \right) + \frac{{t + v - {c^ + }\left( {x,n - 1} \right)}}{{a\left( {{T^n}x} \right)}}} \right)\\
			& = \pi \left( {x,n + \frac{{t + v - {c^ + }\left( {x,n - 1} \right)}}{{a\left( {{T^n}x} \right)}}} \right)\\
			& = \pi \left( {{T^n}x,\frac{{t + v - {c^ + }\left( {x,n - 1} \right)}}{{a\left( {{T^n}x} \right)}}} \right)\\
			& = \left( {{T^n}x,t + v - {c^ + }\left( {x,n - 1} \right)} \right).
		\end{split}
	\end{equation}
	Combining the formula (3.1) and the formula (3.2), we obtain 
	\begin{align*}
		\pi  \circ {\hat \Phi _t} \circ {\pi ^{ - 1}}\left( {x,v} \right) = {\Phi ^a}_t\left( {x,v} \right).
	\end{align*}
	Similarly, When ${c^ - }\left( {x,n} \right) \le t + u < {c^ - }\left( {x,n - 1} \right)$ , 
	\begin{align*}
		\pi  \circ {\hat \Phi _t} \circ {\pi ^{ - 1}}\left( {x,v} \right) = {\Phi ^a}_t\left( {x,v} \right).
	\end{align*}
	When $\left( {x,u} \right)$ is a fixed point, apparently it holds that $\pi  \circ {\hat \Phi _t} \circ {\pi ^ - }\left( {x,v} \right) = {\Phi ^a}_t\left( {x,u} \right)$ is true.
	
	In summary, $\left( {{X^1},\hat \Phi } \right)$ and $\left( {{X^a},{\Phi ^a}} \right)$  are topologically conjugate each other.
	
\end{proof}

\begin{lemma}\label{lem:five}
	Let $\left( {X,\phi } \right)$ and $\left( {Y,\psi } \right)$ be two flows on compact metric spaces $X$ , $Y$ respectively. Assume that they are topologically conjugate each other. The following conclusions hold:
	
	(1) The map ${\pi _ * }:\mu  \in {{\rm M}_{erg,\phi }} \to {\pi _ * }\mu  \in {{\rm M}_{erg,\psi }}$ is bijectic, where ${\pi _ * }\mu \left( B \right) = \mu \left( {{\pi ^{ - 1}}B} \right)$ for all $ B \in {\ B}\left( Y \right)$;
	
	(2) ${h_\mu }\left( \phi  \right) = {h_{{\pi _ * }\mu }}\left( \psi  \right)$.
\end{lemma}

\begin{proof}
	First,we prove the conclusion (1). Obviously ${\pi _ * }\mu $ is a Borel probability measure.
	Take $\mu $ be $\phi $-invariant measure. According to the definition of ${\pi _ * }\mu $, for all $t \in {\rm{R}}$ and $B \in {\rm B}\left( Y \right)$, it holds ${\pi _ * }\mu \left( B \right) = \mu \left( {{\pi ^ - }B} \right)$ and ${\pi _ * }\mu \left( {{\psi _t}B} \right) = \mu \left( {{\pi ^ - } \circ {\psi _t}B} \right)$ . Since $\phi $ and $\psi $ are topologically conjugate each other, and $\mu $ is $\phi $-invariant measure, it follows $\mu \left( {{\pi ^ - } \circ {\psi _t}B} \right) = \mu \left( {{\phi _t} \circ {\pi ^ - }B} \right) = \mu \left( {{\pi ^ - }B} \right)$ . Therefore, there exists ${\pi _ * }\mu \left( B \right) = {\pi _ * }\mu \left( {{\psi _t}B} \right)$, i.e., ${\pi _ * }\mu $ is $\psi $- invariant measure.
	
	Let $\mu $ be $\phi $-ergodic measure. Since $\phi $ and $\psi $ are topologically conjugate each other, it follows that  ${\phi _t} \circ {\pi ^ - }B = {\pi ^ - } \circ {\psi _t}B = {\pi ^ - }B$ if ${\psi _t}B = B$ for all $t \in {\rm{R}}$. And we obtain  $\mu \left( {{\pi ^ - }B} \right) = 1$, because $\mu $  is ergodic. Combining the definition of ${\pi _ * }\mu $ , ${\pi _ * }\mu \left( B \right) = \mu \left( {{\pi ^ - }B} \right) = 1$ or 0 . That is ,${\pi _ * }\mu $ is $\psi $-ergodic measure.
	So, ${\pi _ * }\mu  \in {{\rm M}_{erg,\psi }}$.
	Define ${\pi ^ - }_ * \nu $ for all $\nu  \in {{\rm M}_{erg,\psi }}$, i.e., ${\pi ^ - }_ * \nu \left( A \right) = \nu \left( {\pi A} \right)$ for all $A \in {\mathscr B}\left( X \right)$. In the same way, we obtain ${\pi ^ - }_ * \nu  \in {{\rm M}_{erg,\phi }}$. Since $\pi $ is a homeomorphism, it follows  ${\pi _ * }{\pi ^ - }_ * \nu  = \nu $ .
	Therefore, the map ${\pi _ * }$ is bijective
	
	Conclusion (2) is demonstrated below. Since all of the $\phi $-invariant measures are ${\phi _1}$-invariant measures ,it follows $\left( {X,{\mathscr B}\left( X \right),\mu ,{\phi _1}} \right)$ , $\left( {Y,{\mathscr B}\left( Y \right),{\pi _ * }\mu ,{\psi _1}} \right)$ are two homogeneous preserving-measure systems for $\mu  \in {{\rm M}_{erg,\phi }}$. Since entropy is an equivalent invariant of a discrete dynamical system, it follows $ is obtained{h_\mu }\left( {{\phi _1}} \right) = {h_{{\pi _ * }\mu }}\left( {{\psi _1}} \right)$ . According to the definition of the metric entropy of the flow, it can be obtained ${h_\mu }\left( \phi  \right) = {h_\mu }\left( {{\phi _1}} \right)$ and ${h_{{\pi _ * }\mu }}\left( \psi  \right) = {h_{{\pi _ * }\mu }}\left( {{\psi _1}} \right)$, i.e. ${h_\mu }\left( \phi  \right) = {h_{{\pi _ * }\mu }}\left( \psi  \right)$ .
	
	In summary, Lemma \rm{3.5} is proved.	
\end{proof}

\begin{lemma}\label{lem:nin}
	Let $\left( {X,\phi } \right)$ be a flow on the compact metric space $X$, with only one fixed point, denated as $\Delta $. Let $\left( {X,\hat \phi } \right)$ be the time-changed flow of $\left( {X,\phi } \right)$ with respect to the additive function of $\theta \left( {x,t} \right)$ ,satisfying the following conditions :
	
	(1) $\theta \left( {x,t} \right)$ is continuous with respect to $\left( {x,t} \right)$, for $x \ne \Delta $ ; $\theta \left( {x,t} \right)$ is continuous with respect to $x$ and is strictly increasing with respect to $t$, for $x = \Delta $;
	
	(2) $\theta \left( {x,0} \right) = 0$, $\theta \left( {x,t + s} \right) = \theta \left( {x,t} \right) + \theta \left( {\phi \left( {x,t} \right),s} \right)$;
	
	(3) $\phi \left( {x,t} \right) = \hat \phi \left( {x,\theta \left( {x,t} \right)} \right)$;
	
	(4) $0 < \int\limits_X {\theta \left( {x,1} \right)d\mu }  <  + \infty $, for all $\mu  \in {{\rm M}_{erg,\phi }}$.
	
	Define $\hat \mu  \in {{\rm M}_{erg,\hat \phi }}$ , for all $\mu  \in {{\rm M}_{erg,\phi }}$,
	\begin{align}\label{eq:three}
		{E_{\hat \mu }}\left( f \right) = \frac{1}{{{E_\mu }\left( {\theta \left( {x,1} \right)} \right)}}{E_\mu }\left( {\int\limits_0^{\theta \left( {x,1} \right)} {f\left( {{{\hat \phi }_s}\left( x \right)} \right)ds} } \right),
	\end{align}
	where $f$ is continuous function of $X$.
	
	Then the map $ \wedge :\mu  \in {{\rm M}_{erg,\phi }} \to \mathord{\buildrel{\lower3pt\hbox{$\scriptscriptstyle\frown$}}
		\over \mu }  \in {{\rm M}_{erg,\hat \phi }}$ is bijective.
\end{lemma}

\begin{proof}
	
	In the first step, it is proved that the atomic measure of $\Delta $ is mapped as an atomic measure.
	Let the measure $\mu $ be an atomic measure of $\Delta $ . It follows 
	\begin{align*}
		\hat \mu \left( {\left\{ \Delta  \right\}} \right){\rm{ = }}\frac{{{E_\mu }\left( {\int_0^{\theta \left( {x,1} \right)} {{\chi _\Delta }\left( {{{\hat \phi }_t}x} \right)dt} } \right)}}{{{E_\mu }\left( {\theta \left( {x,1} \right)} \right)}} = \frac{{{E_\mu }\left( {\theta \left( {x,1} \right)} \right)}}{{{E_\mu }\left( {\theta \left( {x,1} \right)} \right)}} = 1.
	\end{align*}
	So, $\hat \mu $ is the atomic measure of $\Delta $ .
	
	Since ${\hat \phi _t}x = {\phi _{\tau \left( {x,t} \right)}}x$ and $\tau \left( {x,t} \right)$ is strictly increasing with respect to $t$, it follows $x$ is not the fixed point of the flow $\hat \phi $ when $x \in X$ is not the fixed point of the flow of $\phi $.
	
	In the second step, we proves that $ \hat \mu  \in {M_{erg,\hat \phi }}$ is well-defined.
	
	Obviously, $\hat \mu $ satisfying the above definition is a measure. And for all $A \in {\mathscr B}\left( X \right)$, it holds
	\begin{align*}
		\hat \mu \left( {\left\{ \Delta  \right\}} \right){\rm{ = }}\frac{{{E_\mu }\left( {\int_0^{\theta \left( {x,1} \right)} {{\chi _\Delta }\left( {{{\hat \phi }_t}x} \right)dt} } \right)}}{{{E_\mu }\left( {\theta \left( {x,1} \right)} \right)}} = \frac{{{E_\mu }\left( {\theta \left( {x,1} \right)} \right)}}{{{E_\mu }\left( {\theta \left( {x,1} \right)} \right)}} = 1.
	\end{align*}
	So, $\hat \mu \left( X \right) = 1$.
	
	In the third step, it is proved that $\hat \mu $ is an invariant measure of the flow $\hat \phi $ .\\
	\textbf{Claim 1.} For all continuous function $f$, $\alpha  > 0$, there is $\int_0^\infty  {{e^{ - \alpha t}}} {E_{\hat \mu }}\left( {f\left( {{{\hat \phi }_t}x} \right)} \right)dt = \frac{1}{\alpha }{E_{\hat \mu }}\left( f \right)$ .
	
	\begin{proof}
		Just verify that $\int_0^\infty  {{e^{ - \alpha t}}} {E_\mu }\left( {\int_0^{\theta \left( {x,1} \right)} {f\left( {{{\hat \phi }_t} \circ {{\hat \phi }_s}x} \right)} ds} \right)dt = \frac{1}{\alpha }{E_\mu }\left( {\int_0^{\theta \left( {x,1} \right)} {f\left( {{{\hat \phi }_s}\left( x \right)} \right)ds} } \right)$.
		
		Simplify the right-hand side of above equation,
		\begin{equation*}
			\begin{split}
				&\int_0^\infty  {{e^{ - \alpha t}}} {E_\mu }\left( {\int_0^{\theta \left( {x,1} \right)} {f\left( {{{\hat \phi }_t} \circ {{\hat \phi }_s}x} \right)} ds} \right)dt \\
				& = {E_\mu }\left( {\int_0^{\theta \left( {x,1} \right)} {\int_0^\infty  {{e^{ - \alpha t}}f\left( {{{\hat \phi }_t} \circ {{\hat \phi }_s}x} \right)dtds} } } \right)\\
				&\mathop {\rm{ = }}\limits^{t + s = l} {E_\mu }\left( {\int_0^{\theta \left( {x,1} \right)} {\int_s^\infty  {{e^{ - \alpha \left( {l - s} \right)}}f\left( {{{\hat \phi }_l}x} \right)dlds} } } \right)\\
				& = {E_\mu }\left( {\int_0^{\theta \left( {x,1} \right)} {{e^{\alpha s}}ds\int_s^\infty  {{e^{ - \alpha l}}f\left( {{{\hat \phi }_l}x} \right)dl} } } \right)\\
				&= {E_\mu }\left\{ {\int_0^{\theta \left( {x,1} \right)} {{e^{ - \alpha l}}f\left( {{{\hat \phi }_l}x} \right)dl\int_0^l {{e^{\alpha s}}ds} }  + \int_{\theta \left( {x,1} \right)}^\infty  {{e^{ - \alpha l}}f\left( {{{\hat \phi }_l}x} \right)dl\int_0^{\theta \left( {x,1} \right)} {{e^{\alpha s}}ds} } } \right\}.
			\end{split}
		\end{equation*}
		
		Since $\int_0^l {{e^{\alpha s}}ds}  = \frac{1}{\alpha }\left( {{e^{\alpha l}} - 1} \right)$, it follows that
		\begin{equation}\label{equ:four}
			\begin{split}
				&\int_0^\infty  {{e^{ - \alpha t}}} {E_\mu }\left( {\int_0^{\theta \left( {x,1} \right)} {f\left( {{{\hat \phi }_t} \circ {{\hat \phi }_s}x} \right)} ds} \right)dt\\
				& = \frac{1}{\alpha }{E_\mu }\left\{ {\int_0^{\theta \left( {x,1} \right)} {\left( {{e^{\alpha l}} - 1} \right){e^{ - \alpha l}}f\left( {{{\hat \phi }_l}x} \right)dl + \left( {{e^{\alpha \theta \left( {x,1} \right)}} - 1} \right)\int_{\theta \left( {x,1} \right)}^\infty  {{e^{ - \alpha l}}f\left( {{{\hat \phi }_l}x} \right)dl} } } \right\}\\
				& = \frac{1}{\alpha }{E_\mu }\left\{ {\int_0^{\theta \left( {x,1} \right)} {f\left( {{{\hat \phi }_l}x} \right)dl}  - \int_0^\infty  {{e^{ - \alpha l}}f\left( {{{\hat \phi }_l}x} \right)dl}  + \int_{\theta \left( {x,1} \right)}^\infty  {{e^{\alpha  - \alpha l}}f\left( {{{\hat \phi }_l}x} \right)dl} } \right\}
			\end{split}.
		\end{equation}
		In addition, we have
		\begin{equation}\label{equ:five}
			\begin{split}
				&{E_\mu }\left( {\int_{\theta \left( {x,1} \right)}^\infty  {{e^{\alpha \theta \left( {x,1} \right) - \alpha l}}f\left( {{{\hat \phi }_l}x} \right)dl} } \right)\\
				&={E_\mu }\left( {\int_1^\infty  {{e^{\alpha \theta \left( {x,1} \right) - \alpha \theta \left( {x,u} \right)}}f\left( {{\phi _u}x} \right)d\theta \left( {x,u} \right)} } \right)\\
				&= {E_\mu }\left( {\int_1^\infty  {{e^{ - \alpha \theta \left( {{\phi _1}x,u - 1} \right)}}f\left( {{\phi _u}x} \right)d\theta \left( {x,u} \right)} } \right)\\
				& = {E_\mu }\left( {\int_1^\infty  {{e^{ - \alpha \theta \left( {{\phi _1}x,u - 1} \right)}}f\left( {{\phi _{u - 1}}x} \right)d\theta \left( {{\phi _{ - 1}}x,u} \right)} } \right)\\
				& = {E_\mu }\left( {\int_0^\infty  {{e^{ - \alpha \theta \left( {x,u} \right)}}f\left( {{\phi _u}x} \right)d\theta \left( {x,u} \right)} } \right)\\
				&= {E_\mu }\left( {\int_0^\infty  {{e^{ - \alpha l}}f\left( {{{\hat \phi }_l}x} \right)dl} } \right)
			\end{split}.
		\end{equation}
		
		Combining the formula \rm{(3.4)} and the formula \rm{(3.5)}, it follows that
		\begin{align*}
			\int_0^\infty  {{e^{ - \alpha t}}} {E_\mu }\left( {\int_0^{\theta \left( {x,1} \right)} {f\left( {{{\hat \phi }_t} \circ {{\hat \phi }_s}x} \right)} ds} \right)dt = \frac{1}{\alpha }{E_\mu }\left( {\int_0^{\theta \left( {x,1} \right)} {f\left( {{{\hat \phi }_s}\left( x \right)} \right)ds} } \right).
		\end{align*}
		So, $\int_0^\infty {{e^{ - \alpha t}}} {E_{\hat \mu }}\left( {f\left( {{{\hat \phi }_t}x} \right)} \right)dt = \frac{1}{\alpha }{E_{\hat \mu }}\left( f \right)$.
		
		In summary, Claim \rm{1} is proved.
	\end{proof}
	
	\textbf{Claim 2.} For all continuous function $f$, ${E_{\hat \mu }}\left( {f\left( {{{\hat \phi }_t}} \right)} \right)$ is continuous with respect to $t$.
	\begin{proof}
		It is clear that
		\begin{center}
			${E_{\hat \mu }}\left( {f\left( {{{\hat \phi }_t}x} \right)} \right) = \frac{{{E_\mu }\left( {\int_0^{\theta \left( {x,1} \right)} {f\left( {{{\hat \phi }_{t + s}}x} \right)ds} } \right)}}{{{E_\mu }\left( {\theta \left( {x,1} \right)} \right)}} = \frac{{{E_\mu }\left( {\int_t^{t + \theta \left( {x,1} \right)} {f\left( {{{\hat \phi }_s}x} \right)ds} } \right)}}{{{E_\mu }\left( {\theta \left( {x,1} \right)} \right)}}$.
		\end{center}
		
		Since $f$ is a continuous function on the compact metric space $X$, the function $f$ is bounded. For all $x \in X$, it holds that $\left| {f\left( x \right)} \right| \le c$, where $c$ is constant.
		For any real number, we obtain
		\begin{equation*}
			\begin{split}
				&\left| {{E_{\hat \mu }}\left( {f\left( {{{\hat \phi }_t}x} \right)} \right) - {E_{\hat \mu }}\left( {f\left( {{{\hat \phi }_l}x} \right)} \right)} \right|\\
				& \le \left| {{E_\mu }\left( {\int_t^{t + \theta \left( {x,1} \right)} {f\left( {{{\hat \phi }_s}x} \right)ds} } \right) - {E_\mu }\left( {\int_l^{l + \theta \left( {x,1} \right)} {f\left( {{{\hat \phi }_s}x} \right)ds} } \right)} \right|\\
				& \le {E_\mu }\left( {\int_l^t {\left| {f\left( {{{\hat \phi }_s}x} \right)} \right|ds} } \right) + {E_\mu }\left( {\int_{\theta \left( {x,1} \right) + l}^{\theta \left( {x,1} \right) + t} {\left| {f\left( {{{\hat \phi }_s}x} \right)} \right|ds} } \right)\\
				& \le \frac{{2c\left( {t - l} \right)}}{{{E_\mu }\left( {\theta \left( {x,1} \right)} \right)}}.
			\end{split}
		\end{equation*}
		Due to $0 < {E_\mu }\left( {\theta \left( {x,1} \right)} \right) < \infty $ , ${E_{\hat \mu }}\left( {f\left( {{{\hat \phi }_t}} \right)} \right)$ is continuous with respect to $t$.
		
		In summary, Claim \rm{2} is proved.	
	\end{proof}
	
	Next prove that $\hat \mu $ is the invariant measure of $\hat \phi $.
	According to Claim\rm{1},
	\begin{center}
		$\int_0^\infty  {{e^{ - \alpha t}}} \left\{ {{E_{\hat \mu }}\left( {f\left( {{{\hat \phi }_t}x} \right)} \right){\rm{ - }}{E_{\hat \mu }}\left( f \right)} \right\}dt = 0$.
	\end{center}
	According to Claim\rm{2}, ${E_{\hat \mu }}\left( {f\left( {{{\hat \phi }_t}x} \right)} \right)$ is continuous. So, for all $t \ge 0$, we obtain 
	\begin{center}
		${E_{\hat \mu }}\left( {f\left( {{{\hat \phi }_t}x} \right)} \right) = {E_{\hat \mu }}\left( f \right)$.
	\end{center}
	Similarly, ${E_{\hat \mu }}\left( {f\left( {{{\hat \phi }_t}x} \right)} \right) = {E_{\hat \mu }}\left( f \right)$, for all $t \ge 0$.
	
	By the equivalence condition defined for invariant measures, $\hat \mu $ is a invariant measure of $\hat \phi $.
	
	In the Third step, prove that $\hat \mu $ is a ergodic measure of $\hat \phi $.\\
	\textbf{Claim 3.} 	Let $f$,\quad $g$ be two continuous functions defined on the compact metric space $X$ with $g > 0$. Then the following limit
	\begin{center}
		$\mathop {\lim }\limits_{t \to \infty } \frac{{\int_0^{\theta \left( {x,t} \right)} {f\left( {{{\hat \phi }_s}x} \right)ds} }}{{\int_0^{\theta \left( {x,t} \right)} {g\left( {{{\hat \phi }_s}x} \right)ds} }} = h\left( x \right)$
	\end{center}
	exists $\hat \mu  - a.e.x \in X$ and $h\left( x \right) = h\left( {{\phi _t}x} \right)$,\quad $\forall t \in {\rm{R}}$.
	
	If $\mu$ is an invariant ergodic measure of $\phi$, then
	\begin{center}
		$h\left( x \right) = \frac{{{E_{\hat \mu }}\left( f \right)}}{{{E_{\hat \mu }}\left( g \right)}}$,\quad $\hat \mu  - a.e.x \in X$.
	\end{center}
	
	\begin{proof}
		Let $F\left( x \right) = \int_0^{\theta \left( {x,1} \right)} {f\left( {{{\hat \phi }_l}x} \right)dl} $,\quad $G\left( x \right) = \int_0^{\theta \left( {x,1} \right)} {g\left( {{{\hat \phi }_l}x} \right)dl} $. It follows that
		\begin{equation*}
			\begin{split}
				& \int_0^{\theta \left( {x,t} \right)} {f\left( {{{\hat \phi }_s}x} \right)ds} \\
				& = \int_0^{\theta \left( {x,1} \right)} {f\left( {{{\hat \phi }_s}x} \right)ds}  + \int_0^{\theta \left( {{\phi _1}x,1} \right)} {f\left( {{{\hat \phi }_s} \circ {\phi _1}x} \right)ds}  + \int_0^{\theta \left( {{\phi _2}x,1} \right)} {f\left( {{{\hat \phi }_s} \circ {\phi _2}x} \right)} ds +  \cdots \\
				& = F\left( x \right) + F\left( {{\phi _1}x} \right) + F\left( {{\phi _2}x} \right) +  \cdots .
			\end{split}
		\end{equation*}
		Thus, applying the Birkhoff's ergodic theorem to $F\left( x \right)$ ,\quad $G\left( x \right)$ and ${\phi _1}$ , we obtain
		\begin{center}
			$\mathop {\lim }\limits_{t \to \infty } \frac{{\int_0^{\theta \left( {x,t} \right)} {f\left( {{{\hat \phi }_s}x} \right)ds} }}{{\int_0^{\theta \left( {x,t} \right)} {g\left( {{{\hat \phi }_s}x} \right)ds} }}$, \quad $\mu  - a.e.x \in X$.
		\end{center}

		Since the zero measure set of $\mu $ is also the zero measure set of $\hat \mu $ \cite{2008TIME}, it follows that
		\begin{center}
			$\mathop {\lim }\limits_{t \to \infty } \frac{{\int_0^{\theta \left( {x,t} \right)} {f\left( {{{\hat \phi }_s}x} \right)ds} }}{{\int_0^{\theta \left( {x,t} \right)} {g\left( {{{\hat \phi }_s}x} \right)ds} }}$,\quad $\hat \mu  - a.e.x \in X$
		\end{center}
		And $h\left( x \right) = h\left( {{\phi _1}x} \right)$, \quad $\mu  - a.e.x \in X$ .
		
		Further, if $\mu $ is the invariant ergodic measure of $\phi $ ,then
		\begin{center}
			$h\left( x \right) = \frac{{{E_{\hat \mu }}\left( f \right)}}{{{E_{\hat \mu }}\left( g \right)}}$,\quad $\hat \mu  - a.e.x \in X$.
		\end{center}
		In summary, Claim \rm{3} is proved.
	\end{proof}
	
	Prove that $\hat \mu $ is the ergodic measure of the flow $\hat \phi $ by contradiction.
	Assume that $\hat \mu $ is not a ergodic measure. There exists an invariant set $A$ of the flow $\phi $, satisfying $\hat \mu \left(A \right) > 0$ and $\hat \mu \left({X\backslash A} \right) > 0$. Let $f$ be a continuous function defined on $X$ and satisfy that for all $x \in a $, \quad $f\left(x \right) = 0$ and ${E_{\hat \mu }}\left( f \right) > 0$. Let $g$ be a positive continuous function defined on $X$. According to Claim \rm{3}, we obtain 
	\begin{center}
		$\mathop {\lim }\limits_{t \to \infty } \frac{{\int_0^{\theta \left( {x,t} \right)} {f\left( {{{\hat \phi }_s}x} \right)ds} }}{{\int_0^{\theta \left( {x,t} \right)} {g\left( {{{\hat \phi }_s}x} \right)ds} }} = \frac{{{E_{\hat \mu }}\left( f \right)}}{{{E_{\hat \mu }}\left( g \right)}}$,\quad $\hat \mu  - a.e.x \in X$.
	\end{center}
	However, when $x \in A$, the above limit is zero, resulting in a contradiction. The hypothesis is not valid. Thus, $\hat \mu $ is the ergodic measure of the flow $\hat \phi $.
	
	In the fourth step, prove that $\wedge $ is bijective.
	
	Take any continuous function $f$ on the compact metric space $X$. If $\int\limits_X {fdm} = \int\limits_X {fd\sigma} $, then $m = \sigma $. So $\ wedge $ is clearly injective. Next we just have to prove that $\wedge $is surjective.	
	
	Define $\ mu '\ in {{\ rm M} _ {erg, \ phi}} $ for all $\ hat \ mu \ in {{\ rm M} _ {erg, \ hat \ phi}} $, satisfying
	\begin{align}\label{equ:six}
		{E_{\mu '}}\left( f \right) = \frac{1}{{{E_{\hat \mu }}\left( {\tau \left( {x,1} \right)} \right)}}{E_{\hat \mu }}\left( {\int\limits_0^{\tau \left( {x,1} \right)} {f\left( {{\phi _t}\left( x \right)} \right)dt} } \right).
	\end{align}
	Similarly, $\mu '$ satisfying the formula \rm{(3.6)} is the Borel probability invariant ergodic measure of flow $\phi $.
	
	Since 
	\begin{equation}\label{equ:seven}
		\begin{split}
			&{E_{\mu '}}\left( {\int\limits_0^{\theta \left( {x,1} \right)} {f\left( {{{\hat \phi }_t}\left( x \right)} \right)dt} } \right)\\
			& = \frac{{{E_{\hat \mu }}\left( {\int\limits_0^{\tau \left( {x,1} \right)} {\int\limits_0^{\theta \left( {x,1} \right)} {f\left( {{{\hat \phi }_s}\left( {{\phi _t}\left( x \right)} \right)} \right)} } dsdt} \right)}}{{{E_{\hat \mu }}\left( {\tau \left( {x,1} \right)} \right)}}\\
			&{\rm{ = }}\frac{{{E_{\hat \mu }}\left( {\int\limits_0^1 {f\left( {{{\hat \phi }_t}\left( x \right)} \right)} ds} \right)}}{{{E_{\hat \mu }}\left( {\tau \left( {x,1} \right)} \right)}}\\
			&{\rm{ = }}\frac{{{E_{\hat \mu }}\left( f \right)}}{{{E_{\hat \mu }}\left( {\tau \left( {x,1} \right)} \right)}}\\
			&{\rm{ = }}\frac{{{E_{\hat \mu }}\left( {\int\limits_0^{\theta \left( {x,\tau \left( {x,1} \right)} \right)} {f\left( {{{\hat \phi }_t}\left( x \right)} \right)} ds} \right)}}{{{E_{\hat \mu }}\left( {\tau \left( {x,1} \right)} \right)}}
		\end{split}
	\end{equation}
	and
	\begin{equation}\label{equ:eig}
		\begin{split}
			&{E_{\mu '}}\left( {\theta \left( {x,1} \right)} \right)\\
			&  = \frac{{{E_{\hat \mu }}\left( {\int\limits_0^{\tau \left( {x,1} \right)} {\theta \left( {{\phi _t}\left( x \right),1} \right)dt} } \right)}}{{{E_{\hat \mu }}\left( {\tau \left( {x,1} \right)} \right)}}\\
			&= \frac{{{E_{\hat \mu }}\left( {\theta \left( {x,\tau \left( {x,1} \right)} \right)} \right)}}{{{E_{\hat \mu }}\left( {\tau \left( {x,1} \right)} \right)}}\\
			&= \frac{1}{{{E_{\hat \mu }}\left( {\tau \left( {x,1} \right)} \right)}}
		\end{split},
	\end{equation}
	it follows that
	\begin{center}
		\[{E_{\hat \mu }}\left( f \right) = \frac{1}{{{E_{\mu '}}\left( {\theta \left( {x,1} \right)} \right)}}{E_{\mu '}}\left( {\int\limits_0^{\theta \left( {x,1} \right)} {f\left( {{{\hat f}_t}\left( x \right)} \right)dt} } \right)\]
	\end{center}
	with combination of the formula \rm{(3.7)} and the formula \rm{(3.8)}. 
	Thus, $\wedge $ is surjective.
	
	To sum up, Lemma \rm{3.6} is proved.	
	
\end{proof}

\begin{lemma}\label{lem:ten}
	Let $\left({X,\phi} \right)$ be the flow on the compact metric space $X$, and $\left({\hat X,\hat \phi} \right)$ be the time-changed flow induced by the additive function $\theta \left({x,t} \right)$ satisfying the conditions of the above theorem. For all $\mu  \in {{\rm M}_{erg,\phi }}$, we have 
	\begin{center}
		${h_{\hat \mu }}\left( {\hat \phi } \right) = \frac{{{h_\mu }\left( \phi  \right)}}{{{E_\mu }\left( {\theta \left( {x,1} \right)} \right)}}$.
	\end{center}
\end{lemma}

\begin{proof}
	It is only necessary to show that the above equality holds for any non-trivial measure $\mu $.According to the knowledge of the time-changed flow, the regular set $\hat X = \left\{ {x \in X|\theta \left( {x,t} \right) > 0,\forall t > 0} \right\}$, and $\hat X = X$ .
	
	According to ref, for all fixed $t \in {\rm{R}}$, we can get 
	\begin{center}
		${h_{\hat \mu }}\left( {{{\hat \phi }_t}} \right)\hat \mu \left( {\hat X} \right){E_\mu }\left( {\theta \left( {x,1} \right)} \right) = {h_\mu }\left( {{\phi _t}} \right)\mu \left( X \right)$.
	\end{center}
	Therefore, from the definition of the measure entropy of flow,
	\begin{center}
		${h_{\hat \mu }}\left( {\hat \phi } \right) = \frac{{{h_\mu }\left( \phi  \right)}}{{{E_\mu }\left( {\theta \left( {x,1} \right)} \right)}}$.
	\end{center}
	
	In summary, Lemma \rm{3.7} is proved.
\end{proof}

\textbf{The proof of the Theorem A.} Based on the variational principle of flow, there exists $\ mu \ in {{\ rm M} _ {erg, {\ Phi ^ 1}}} $, such that ${h_\mu }\left( {{\Phi ^1}} \right) = c > 0$ . Define a constant function on the compact metric space ${X^1}$, $a\left({x,u} \right) = \frac{c}{b}$, $\forall \left({x,u} \right) \in {X^1}$. Clearly, $a\left({x,u} \right)$ is continuous with respect to $\left({x,u} \right)$. Define an additive function with $a\left({x,u} \right)$
\begin{center}
	$\theta :{X^1} \times {\rm{R}} \to {\rm{R}}$, $\theta \left( {x,u,t} \right) = \int\limits_0^t {a\left( {{\Phi ^1}_t\left( {x,u} \right)} \right)ds} $.
\end{center}
Note that $\left({{X^1},{{\hat \Phi}^1}} \right)$is the time-changed flow of  $\left({{x ^1},{\Phi ^1}} \right)$ induced by the additive function $\theta \left({X,u,t} \right)$ 

Since $0 < \int\limits_{{X^1}} {\theta \left( {x,u,1} \right)} d\mu  = \int\limits_{{X^1}} {a\left( {x,u} \right)} d\mu  = \frac{c}{b} <  + \infty $ ,so that applying Lemma \rm{3.6} to the measure $\mu $, there exists $\hat \mu \in {{\rm M}_{erg,{{\hat \Phi}^1}}}$satisfying the equation \rm{(3)}. According to Lemma \rm{3.7}, it can be obtained
\begin{center}
	${h_{\hat \mu }}\left( {{{\hat \Phi }^1}} \right) = \frac{{{h_\mu }\left( {{\Phi ^1}} \right)}}{{{E_\mu }\left( {\theta \left( {x,u,1} \right)} \right)}} = \frac{c}{{{\raise0.7ex\hbox{$c$} \!\mathord{\left/
					{\vphantom {c b}}\right.\kern-\nulldelimiterspace}
				\!\lower0.7ex\hbox{$b$}}}} = b$.
\end{center}

According to Lemma \rm{3.4}, the time-changed flow $\left( {{X^1},\hat \Phi } \right)$  of the flow $\left( {{X^1},{\Phi ^1}} \right)$ and the flow $\left( {{X^{\frac{c}{b}}},{\Phi ^{\frac{c}{b}}}} \right)$ are topologically conjugate, where the homeomorphism  $\pi :{X^1} \to {X^{\frac{a}{b}}}$ is defined as
\begin{center}
	$\pi \left( {x,u} \right) = \left\{ {\begin{array}{*{20}{c}}
			{\left( {x,u\frac{c}{b}} \right),}&{x \in {X_ * },0 \le u < 1}\\
			{fixed-point,}&{fixed-point }
	\end{array}} \right.$.
\end{center}
According to Lemma\rm{3.5}, there exists ${\pi _ * }\hat \mu  \in {{\rm M}_{erg,{\Phi ^{\frac{c}{b}}}}}$ such that  ${h_{{\pi _ * }\hat \mu }}\left( {{\Phi ^{\frac{c}{b}}}} \right) = {h_{\hat \mu }}\left( {{{\hat \Phi }^1}} \right) = b$.

In summary, Theorem \rm{A} is proved.

\textbf{The proof of the Theorem B.} Define a constant function on the compact metric space ${X^1}$, $a\left({x,u} \right) = \frac{{h\left(T \right)}}{b}$, $\forall \left({x,u} \right) \in {X^1}$. Clearly, $a\left({x,u} \right)$ is continuous with respect to $\left({x,u} \right)$. Define an additive functions with $a\left({x,u} \right)$
\begin{center}
	$\theta :X^1 \times \rm{R} \rightarrow \rm{R}$,
	$\theta \left( {x,u,t} \right) = \int\limits_0^t {a\left( {{\Phi ^1}_t\left( {x,u} \right)} \right)ds} $.
\end{center}
Note that $\left({{X^1},{{\hat \Phi}^1}} \right)$ is the time-changed flow of $\left({{x ^1},{\Phi ^1}} \right)$ induced by the additive function $\theta \left({X,u,t} \right)$ . For any non-atomic Borel probability invariant ergodic measure $\mu $of the flow ${\Phi ^1}$, we have 
\begin{center}
	$0 < \int\limits_{{X^1}} {\theta \left( {x,u,1} \right)} d\mu  = \int\limits_{{X^1}} {a\left( {x,u} \right)} d\mu  = \frac{{h\left( T \right)}}{b} <  + \infty $.
\end{center}
Applying Lemma \rm{3.6}, it follows that
\begin{center}
	${h_{\hat \mu }}\left( {\hat \Phi } \right) = \frac{{{h_\mu }\left( \Phi  \right)}}{{{E_\mu }\left( {\theta \left( {x,u,1} \right)} \right)}}$,
\end{center}
where $\mu $ ,$\hat \mu $ satisfy the equation \rm{(3)}. According to the variational principle of flow, it follows that
\begin{equation*}
	\begin{split}
		h\left( {{{\hat \Phi }^1}} \right) &= \mathop {\sup }\limits_{\hat \mu  \in {{\rm M}_{erg,{{\hat \Phi }^1}}}} {h_{\hat \mu }}\left( {\hat \Phi } \right)\\
		&= \mathop {\sup }\limits_{\mu  \in {{\rm M}_{erg,{\Phi ^1}}}} \frac{{{h_\mu }\left( {{\Phi ^1}} \right)}}{{{E_\mu }\left( {\theta \left( {x,u,1} \right)} \right)}}\\
		& = \frac{{h\left( {{\Phi ^1}} \right)}}{{{E_\mu }\left( {\theta \left( {x,u,1} \right)} \right)}}\\
		& = \frac{{h\left( T \right)}}{{{\raise0.7ex\hbox{${h\left( T \right)}$} \!\mathord{\left/
						{\vphantom {{h\left( T \right)} b}}\right.\kern-\nulldelimiterspace}
					\!\lower0.7ex\hbox{$b$}}}}\\
		& = b.
	\end{split}
\end{equation*}

According to Lemma \rm{3.4}, the flow $\left( {{X^{\frac{{h\left( T \right)}}{b}}},{\Phi ^{\frac{{h\left( T \right)}}{b}}}} \right)$ and the time-changed flow $\left( {{X^1},\hat \Phi } \right)$ of $\left( {{X^1},{\Phi ^1}} \right)$ are topologically conjugate. The homeomorphism $\pi :{X^1} \to {X^{\frac{{h\left( T \right)}}{b}}}$ is defined as
\begin{center}
	$\pi \left( {x,u} \right) = \left\{ {\begin{array}{*{20}{c}}
			{\left( {x,u\frac{{h\left( T \right)}}{b}} \right),}&{x \in {X_ * },0 \le u < 1}\\
			{fixed-point,}&{fixed-point}
	\end{array}} \right.$.
\end{center}
According to Lemma \rm{3.5}, we have $h\left( {{\Phi ^{\frac{{h\left( T \right)}}{b}}}} \right) = h\left( {{{\hat \Phi }^1}} \right) = b$.

In summary, Theorem \rm{B} is proved.

\section{The flow with no entropy explosion}

\subsection{Construction of the example}

Now we are the structure of the example is given in detail.

Step 1, Selecting a unidirectional infinite sequence $y = \left({y\left(i \right)} \right) \in \Sigma \left(2 \right)$.

Make some notations. Let $A$, $B$ be finite sequences of $m $ and $n$ 0's or 1's, respectively. Write as $ \left| A \right| = m$, $\left| B \right| = n $. If the finite sequence $A$ and $B$ are combined end to end, the resulting new sequence is denoted as $AB$. Apparently, $\left| {AB} \right| = \left| A \right| + \left| B \right| = m + n $.

For any positive integer $n$, let $P\left(n \right)$ be a finite sequence of 0's or 1's of length $n$, and $O\left(n \right)$ be a finite sequence of all zeros of length $n$. Consider in the following an infinite matrix made of finite sequences:
\begin{center}
	\[\left( {\begin{array}{*{20}{c}}
			{P\left( 1 \right)O\left( 1 \right)P\left( 1 \right)}&{P\left( 1 \right)O\left( 2 \right)P\left( 1 \right)}& \cdots &{P\left( 1 \right)O\left( m \right)P\left( 1 \right)}& \cdots \\
			{P\left( 2 \right)O\left( 4 \right)P\left( 2 \right)}&{P\left( 2 \right)O\left( 5 \right)P\left( 2 \right)}& \cdots &{P\left( 2 \right)O\left( {3 + m} \right)P\left( 2 \right)}& \cdots \\
			\vdots & \vdots & \vdots & \vdots & \vdots \\
			{P\left( n \right)O\left( {{n^2}} \right)P\left( n \right)}&{P\left( n \right)O\left( {{n^2} + 1} \right)P\left( n \right)}& \cdots &{P\left( n \right)O\left( {{n^2} - 1 + m} \right)P\left( n \right)}& \cdots \\
			\vdots & \vdots & \vdots & \vdots & \vdots
	\end{array}} \right)\]
\end{center}

Denote each finite sequence in an infinite matrix as ${Q_i}$, $i > 0$ in the following order:

\begin{center}
	\[\left( {\begin{array}{*{20}{c}}
			{{Q_1}}&{{Q_2}}&{{Q_4}}&{{Q_7}}&{{Q_{11}}}& \cdots \\
			{{Q_3}}&{{Q_5}}&{{Q_8}}&{{Q_{12}}}& \cdots & \cdots \\
			{{Q_6}}&{{Q_9}}&{{Q_{13}}}& \cdots & \cdots & \cdots \\
			{{Q_{10}}}&{{Q_{14}}}& \cdots & \cdots & \cdots & \cdots \\
			{{Q_{15}}}& \cdots & \cdots & \cdots & \cdots & \cdots \\
			\cdots & \cdots & \cdots & \cdots & \cdots & \cdots
	\end{array}} \right)\]
\end{center}
Obviously, when each $P\left(n \right)$ is determined, ${Q_i}$ is also completely determined.

Let ${A_1} = \left(1 \right)$, $P\left(1 \right) = {A_1}$, Then ${Q_1} = P\left(1 \right)O\left(1 \right)P\left(1 \right)$. And all ${Q_i}$ in the first row of the infinite matrix above is determined. Let ${A_2} = {A_1}O\left( {{{\left| {{A_1}} \right|}^2}} \right){Q_1}$ . Apparently, $\left| {{A_2}} \right| \ge 2$ . Taking $P\left(2 \right)$ to be the first two terms of ${A_2}$, then all ${Q_i}$ in the second row of the infinite matrix above are determined. By induction of $n \ge 3$, it is assumed that ${A_{n-1}}$, $P\left( n-1 \right)$, ${Q_{n-1}}$ are all determined. Let  ${A_n} = {A_n-1}O\left( {{{\left| {{A_{n-1}}} \right|}^2}} \right){Q_{n-1}}$. Apparently, $\left| {{A_n}} \right| \ge n$. Taking $P\left(n \right)$ to be the first $n$ term of ${A_n}$, then ${Q_n}$ is determined by some $P\left(j \right)$, $1 \le j \le n$.

From the structure, the first $\left| {{A_{n - 1}}} \right|$ terms of $A_n$ is $A_{n - 1}$. So the following limit exists:

\begin{center}
	$y = \mathop {\lim }\limits_{n \to \infty } \left( {{A_n}000 \cdots } \right) \in \Sigma \left( 2 \right)$.
\end{center}

It clearly follows that $y$ is an aperiodic reply point of $T$. That is, for any $\varepsilon > 0$, there exists $n > 0$ such that
\begin{center}
	$d\left( {y,{T^n}y} \right) < \varepsilon $.
\end{center}

Step 2, Constructing a subsystem of the symbol system $\left({\Sigma \left(2 \right),T} \right)$.

Let $Y = \omega \left({y,T} \right)$ be the $\omega$- limit set of $y$ with respect to $T$. So $Y$  is the compact invariant subset of $T$, and  $T{\rm{ = }}T{|_Y}:Y \to Y$ form subsystem. 

\textbf{Claim.} $T{|_Y}:Y \to Y$ is a homeomorphism.
\begin{proof}
	First, $TY \subset Y$ because $T$ is continuous.\\
	Any given $x \in Y$, there exists ${n_i} \to + \infty $ such that ${T^{{n_i}}}y \to x$. There exists ${n_i} \to  + \infty $ for all $x \in Y$ such that ${T^{{n_i}}}y \to x$. Due to the compactness, we obtain that $\left\{ {{T^{{n_{{i_j}}} - 1}}y} \right\}$ is a converging subsequence of $\left\{ {{T^{{n_i} - 1}}y} \right\}$. Denote as ${T^{{n_{{i_j}}} - 1}}y \to z$. Obviously, $z \in Y$ is unique. It follows that $T\left( {{T^{{n_{{i_j}}} - 1}}y} \right) \to Tz$. By the uniqueness of the limit, it follows that $Tz = x$. Denote as $z = {T^ -}x$. Therefore, $T$ is a one-to-one mapping and ${T^ -}$ is continuous.\\
	In summary, the claim is proved.
\end{proof}

Step 3, Suspending $\left( {Y,T} \right)$ into a flow.

The discrete system $\left( {Y,T} \right)$  has a unique fixed point (denoted ) $\left\{ 0 \right\} \in Y$, which is a bidirectional infinite sequence with each position 0. Let ${Y_ * } = Y\backslash \left\{ 0 \right\}$. Then ${Y_ * }$ is a locally compact space. Let $\gamma :{Y_ * } \to \left( {0, + \infty } \right)$ be a positive continuous function. Then $\left( {{Y_ * },T} \right)$ can be suspended with $\gamma $ to get a flow. Specifically, establish an equivalence relationship on $\left\{ {\left( {y,u} \right)|0 \le u \le \gamma \left( y \right),y \in {Y_ * }} \right\}$: 
\begin{center}
	$\left( {y,\gamma \left( y \right)} \right) \sim \left( {Ty,0} \right)$.
\end{center}
We obtain the quotient space ${Y_ * }^\gamma $. The flow ${\psi ^\gamma }:{Y_ * }^\gamma  \times {\rm{R}} \to {Y_ * }^\gamma $ is defined as
\begin{center}
	${\psi ^\gamma }_t\left( {y,u} \right) = \left( {y,u + t} \right)$,\quad $ - u \le t < \gamma \left( y \right) - u$.
\end{center}
In this case, ${Y_ * }^\gamma $  is the metric space of local compactness and ${\psi ^\gamma }$ is the flow defined on the space ${Y_ * }^\gamma $.

Let ${Y^\gamma } = {Y_ * }^\gamma  \cup \left\{ {{y_\infty }} \right\}$  is the single-point compactification of ${Y_ * }^\gamma $. Then ${Y^\gamma }$  is the compactness metric space. And $\left( {{y_n},{u_n}} \right) \to {y_\infty }$ on the space ${Y^\gamma }$  if and only if there is ${y_n} \to 0$ on the space $Y$. Put ${\psi ^\gamma }_t\left( {{y_\infty }} \right) = {y_\infty }$. For all $t \in {\rm{R}}$, we expand the flow ${\psi ^\gamma }$  on ${Y_ * }^\gamma $ to the flow on  ${Y^\gamma }$ , and denote it ${\Psi ^\gamma }$.

\subsection{The proof of Theorem C and Theorem D}

Let ${\Psi ^1}$ be the flow constructed by the roof function $\gamma  = 1$. Now give the main theorem of this chapter.

\textbf{Theorem C.} The flow ${\Psi ^1}$ is topologically transitive. \\
\textbf{Theorem D.} The flow ${\Psi ^1}$ has only one invariant ergodic measure, which is supported on the fixed point ${y_\infty}$.

According to Theorems C and D, it is concluded that no entropy explosion can occur for the flow ${\Psi ^1}$ constructed above.

Before proving Theorems C and D, first show the topological complexity and statistical triviality of the subsystem $\left({Y,T} \right)$, see the reference\cite{1991Point}.

\begin{lemma}\label{lem:ee}
	Let $f:X \to X$ be a continuous mapping on a compact metric space, and the system $\left({X,f} \right)$ be topologically transitive. Then $f$ is topologically mixing if and only if for any $\varepsilon > 0$, there exists $N > 0$ such that ${f^n}\left( {V\left( {x,\varepsilon } \right)} \right) \cap V\left( {x,\varepsilon } \right) \ne \emptyset $ when $n \ge N$, where $X = \omega \left( {x,f} \right)$ ,\quad $V\left( {x,\varepsilon } \right)$ refers to the open ball with $x$as the center and $\varepsilon $as the radius.
\end{lemma}
\begin{lemma}\label{lem:er}
	$\left( {Y,T} \right)$ is topologically mixing.
\end{lemma}
\begin{lemma}\label{lem:es}
	$\left( {Y,T} \right)$ has only one invariant ergodic measure, which is supported on the fixed point $\left\{ 0 \right\}$.
\end{lemma}

The topological complexity and statistical tribality of the flow ${\Psi ^1}$ are discussed next.

\textbf{The proof of Theorem C.} It is known that the continuous flow ${\Psi ^1}$ is defined on the compact metric space ${Y^\gamma}$. Just prove that ${Y^\gamma} = \omega \left({\left({y,0} \right),{\Psi ^1}} \right)$.
Let ${y_\infty }$ be not $\left( {y',u} \right) \in {Y^\gamma }$. There exists ${n_i} \to  + \infty $ such that ${T^{{n_i}}}y \to y'$ because $y' \in Y = \omega \left( {y,T} \right)$. Consider ${\Psi ^1}_{{n_i} + u}\left( {y,0} \right) \sim \left( {{T^{{n_i}}}y,u} \right) \in {Y^\gamma }$. It follows that ${\Psi ^1}_{{n_i} + u}\left( {y,0} \right) \to \left( {y',u} \right)$.

Let $\left( {y',u} \right) = {y_\infty } \in {Y^\gamma }$. There exists ${n_i} \to  + \infty $ such that ${T^{{n_i}}}y \to \left\{ 0 \right\}$ because $\left\{ 0 \right\} \in Y = \omega \left( {y,T} \right)$. Consider ${\Psi ^1}_{{n_i}}\left( {y,0} \right) \sim \left( {{T^{{n_i}}}y,0} \right) \in {Y^\gamma }$. It follows that ${\Psi ^1}_{{n_i}}\left( {y,0} \right) \to {y_\infty }$ .

Therefore, ${Y^\gamma } = \omega \left( {\left( {y,0} \right),{\Psi ^1}} \right)$. That is the flow ${\Psi ^1}$ is topologically transitive.

In summary, Theorem C is proved.
\begin{lemma}\label{lem:es}
	Let ${\gamma _0} = \mathop {\inf }\limits_{y \in {Y_ * }} \gamma \left( y \right) > 0$. Then there exists nontrivial measure (i.e. not  the atomic measure at $\left\{ 0 \right\}$  ) $\mu  \in {{\rm M}_{erg,T}}$ for every given nontrivial measure ( i.e. not  the atomic measure at ${y_\infty }$), such that 
	\begin{center}
		${E_{\bar \mu }}\left( f \right) = \frac{1}{{{E_\mu }\left( \gamma  \right)}}{E_\mu }\left( {\int\limits_0^{\gamma \left( x \right)} {f\left( {x,t} \right)dt} } \right)$,
	\end{center}
	for all continuous function $f$ on ${Y^\gamma}$, where 
	\begin{center}
		${E_{\bar \mu }}\left( f \right) = \int {fd\bar \mu } $ ,\quad ${E_\mu }\left( \gamma  \right) = \int {\gamma d\mu } $,\\
		${E_\mu }\left( {\int\limits_0^{\gamma \left( x \right)} {f\left( {x,t} \right)dt} } \right) = \int {\left( {\int\limits_0^{\gamma \left( x \right)} {f\left( {x,t} \right)dt} } \right)d\mu } $.
	\end{center}
\end{lemma}
\begin{proof}
	Following the reference \cite{1980A, R1987Ergodic}, it is easy to prove the conclusion.
\end{proof}

\textbf{The proof of Theorem D.} (Contradiction)

Assume that $\bar \mu $ is another invariant ergodic measure that is not supported at a fixed point. According to Lemma \rm{4.4}, there exists atomic measures that are not $\left\{0 \right\}$ such that for any continuous function $f$ on ${Y^\gamma}$,
\begin{center}
	${E_{\bar \mu }}\left( f \right) = {E_\mu }\left( {\int\limits_0^1 {f\left( {x,t} \right)dt} } \right)$,
\end{center}
where ${E_{\bar \mu }}\left( f \right) = \int {fd\bar \mu } $,\quad ${E_\mu }\left( {\int\limits_0^1 {f\left( {x,t} \right)dt} } \right) = \int {\left( {\int\limits_0^1 {f\left( {x,t} \right)dt} } \right)d\mu } $.

It follows that the subsystem $\left({Y,T} \right)$ has non-trivial invariant ergodic measure, which is obviously contradictory to the subsystem having only one invariant ergodic measure supported at the unique fixed point. Therefore, the hypothesis is not valid.

In conclusion, Theorem D is proved.

In fact, from Lemma \rm{4.4}, the suspending flow ${\Psi ^\gamma }$ is constructed by any given positive continuous function $\gamma $ which satisfies ${\gamma _0} = \mathop {\inf }\limits_{y \in {Y_ * }} \gamma \left( y \right) > 0$, contains only one invariant ergodic measure supported at the fixed point.

From the point of view of entropy, since a flow ${\Psi ^\gamma }$ contains only one invariant ergodic measure supported at a fixed point, its entropy is zero. This example also shows that entropy explosion does not happen in any flow, and entropy degeneracy happens independently.

\textbf{Acknowledgments.} 	The author sincerely thanks Professor Wenxiang Sun and Professor YunHua Zhou for their discussion and help.
\bibliography{REFFERENCE}

\end{document}